\documentclass{amsart}

\usepackage{graphicx}
\usepackage{systeme}
\usepackage{setspace}
\newtheorem{theorem}{Theorem}[section]
\newtheorem{lemma}[theorem]{Lemma}

\theoremstyle{definition}

\newtheorem{corollary}{Corollary}[theorem]

\theoremstyle{remark}

\numberwithin{equation}{section}

\begin{document}

\title{Gaps of powers of consecutive primes and some consequences }

\author{Douglas Azevedo}

\author{Tiago Reis}

\address{UTFPR-CP. Av. Alberto Carazzai 1640, centro, caixa postal 238,  86300000, Cornelio Procopio, PR, Brasil.}
%\curraddr{Department of Mathematics and Statistics,
%Case Western Reserve University, Cleveland, Ohio 43403}
\email{dgs.nvn@gmail.com}

%\author{Author Two}
%\address{Mathematical Research Section, School of Mathematical Sciences,
%Australian National University, Canberra ACT 2601, Australia}

%    General info

%\date{January 1, 2001 and, in revised form, June 22, 2001.}


%}

\begin{abstract}
Let $\{q_n\}$ be a sequence of positive numbers and  $x\in\mathbb{R}$. In this note we prove that the inequality
$$q_n p_{n+1}^{x}-q_{n+1}p_{n}^{x}<p_{n}^{x}p_{n+1}^{x-1}, $$
holds for infinitely many values of $n$.  
As it is shown, the key ingredient to obtain this behaviour  is a consequence  of an extension of the Kummer's characterization of convergent series of positive terms.
\end{abstract}

\maketitle

\section{Background and main result}
\label{sec1}

 The behaviour of the prime numbers is one of the most interesting issues in mathematics and  many great mathematicians have been working on this subject, for instance, we indicate \cite{gold,gold1,maynard,polymath8,zhang}, and references therein. In particular, the investigation of gaps between consecutive prime numbers, that is, the behaviour of the sequence $\{g_n\}$, defined as  $g_n = p_{n+1} - p_n$, for all positive integer $n$, in which, $p_n$ denotes the $n$th prime number, is among one of the most important unsolved problems in number theory. 
 
 In this note we present some information about the sequence $\{p_{n+1}^{x}-p_{n}^{x}\}$, in which $x$ is a positive real number. Our main result is as follows
 
\begin{theorem}\label{main}
Let $x\in\mathbb{R}$.
For any sequence of positive terms $\{q_n\}$, 
%any positive integer
%$N$ there exists a $n>N$ such that
the inequality
$$q_n p_{n+1}^{x}-q_{n+1}p_{n}^{x}<\frac{p_{n}^{x}}{p_{n+1}^{1-x}} $$
holds for infinitely many values of $n$.
\end{theorem}

It is clear that Theorem \ref{main} provides a general inequality involving powers of  prime numbers and by exploring this inequality, for suitable choices of $x\in\mathbb{R}$ we are able to obtain interesting information about the sequence $\{p_{n+1}^{x}-p_{n}^{x}\}$. Note also that this result is related to the main result of \cite{azevedo1}, which deals with the case $x=1$. 

The method to prove Theorem \ref{main} depends on an extension of  the Kummer's test for convergence of series of positive terms. This  test is actually a   theoretical characterization of convergent  series of positive terms, that is, it  provides necessary and sufficient conditions that ensures convergence of series of positive terms. For more information about the original Kummer's test we refer to \cite{Tong:2004}.  

Another important result that plays a fundamental role in the proof of Theorem \ref{main} is the well known  divergence of the series of the reciprocal of the primes numbers. Roughly speaking,  the  idea behind the proof of Theorem \ref{main} is to combine the divergence of this series with  the contrapositive argument of an  extension of the Kummer's test.

Let us now present the results that will be needed to prove our main result and then extract some interesting information about  gaps of powers of prime numbers.

%Let us start with the following consequences of the Prime Number Theorem
% (see \cite[p. 80]{apostol}). 

%\begin{lemma}\label{PNT}
%$$(i)\,\,\lim_{n\to\infty}\frac{p_n}{n\log(n)}=1,$$
%and, consequently,
%$$(ii)\,\,\lim_{n\to\infty}\frac{\log(p_n)}{\log(n)}=1.$$

%\end{lemma}
%\begin{proof}
%First equality is a consequence of  the Prime Number Theorem.
%The second equality is an imediate consequence of 
%$$\lim_{n\to\infty} \frac{\log(n)}{  \log(n\log(n))}=1. $$

%\end{proof}

We start by presenting an extension of the Kummer's test which characterizes the convergence of series of positive terms.

\begin{lemma}\label{Lim}
Let  $\{a_n\}$ and $\{b_n\}$ are sequences of positive terms.
The series $\sum a_n b_n$ converges if and only if there exist a sequence $\{q_n\}$ of positive terms and a positive integer $N$ such that
$$q_n\frac{a_n}{a_{n+1}}-q_{n+1}\geq b_{n+1}, $$
for $n>N$.
\end{lemma}

\begin{proof}
 Let us show that $\sum_{n=1}^{\infty}a_{n}b_{n}$ converges. For this, note that the condition 
$$q_{n}\frac{a_{n}}{a_{n+1}}-p_{n+1}\geq b_{n+1},\,\,n\geq N $$
implies that 
\begin{equation}\label{auxx}
a_{n}q_{n}\geq a_{n+1}(q_{n+1}+ b_{n+1}),\,\,n\geq N.
\end{equation}

Hence, for any fixed $k \geq 0$ we have that
\begin{eqnarray*}
%a_Nq_N &\geq & a_{N+1}q_{N+1}+a_{N+1}b_{N+1}, \\
%a_{N+1}q_{N+1} &\geq & a_{N+2}q_{N+2}+a_{N+2}b_{N+2}, \\
%&\vdots & \\
a_{N+k}q_{N+k} &\geq & a_{N+k+1}q_{N+k+1}+a_{N+k+1}b_{N+k+1},
\end{eqnarray*}
as so, 
\begin{eqnarray*}
a_{N}q_{N} +\sum_{i=1}^k a_{N+i}q_{N+i} & \geq& \sum_{i=1}^{k+1}a_{N+i}b_{N+i}+\sum_{i=1}^k a_{N+i}q_{N+i}  +a_{N+k+1}q_{N+k+1}.
\end{eqnarray*}
Therefore we obtain 
$$a_Nq_N\geq  \sum_{i=1}^{k+1}a_{N+i}b_{N+i}  +a_{N+k+1}q_{N+k+1} \geq \sum_{i=1}^{k+1}a_{N+i}b_{N+i}>0, $$
for all integer $k\geq 0$. This implies the convergence of $\sum_{n=1}^{\infty}a_{n}b_{n}$.

Conversely, suppose that 
$S:=\sum_{n=1}^{\infty}a_{n}b_{n}$ and define 
\begin{equation}\label{pn}
q_{n}=\,\frac{S-\sum_{i=1}^{n}a_{i}b_{i}}{a_{n}},\,\,n\geq N.
\end{equation} 
 For this $\{q_{n}\}$, clearly $q_{n}>0$ for all $n\geq 1$ and  it is easy to check that
\begin{eqnarray*}
q_{n}\frac{a_{n}}{a_{n+1}}-q_{n+1}= b_{n+1}, \,\,n\geq 1.
\end{eqnarray*}

\end{proof}

As we have mentioned before, the key argument in the proof of Theorem \ref{main} is the
 contrapositive of Lemma \ref{Lim}, which is presented below. 

\begin{lemma}\label{contrap}
Let  $\{a_n\}$ and $\{b_n\}$ are sequences of positive terms.
The series $\sum a_n b_n$ diverges if and only if for each $\{q_n\}$ of positive terms we have that: for all  positive integer $N$ there exist $n>N$ such that
$$q_n\frac{a_n}{a_{n+1}}-q_{n+1}< b_{n+1}. $$

\end{lemma}

\section{Proof Theorem \ref{main} and some consequences}

In this section we prove Theorem \ref{main} and present some interesting 
consequences.

\begin{proof}[Proof of Theorem \ref{main}] It is clear that
\begin{eqnarray}
\sum_{n=1}^{\infty}\frac{1}{p_n}&=&\sum_{n=1}^{\infty}\frac{1}{p_n^{x}}p_{n}^{x-1}
\end{eqnarray}
holds for any real number $x$. Since $\sum \frac{1}{p_n}$ diverges, a direct application of  Lemma \ref{contrap} with $a_n=\frac{1}{p_{n}^{x}}$ and $b_n=p_{n}^{x-1}$ give us that, for all $\{q_n\}$ of positive terms, we have that for all positive integer $N$ there exists $n>N$ such that 
$$q_{n}\frac{1/p_{n}^{x}}{1/p_{n+1}^{x}}-q_{n+1}<p_{n+1}^{x-1}.$$
That is, for every sequence $\{q_n\}$ of positive numbers, the inequality
$$q_n p_{n+1}^{x}-q_{n+1}p_{n}^{x}<\frac{p_{n}^{x}}{p_{n+1}^{1-x}}, $$
holds for infinitely many values of $n$.
\end{proof}

Let us now extract some consequences 
 of Theorem \ref{main}.

\begin{corollary}\label{coroaux} Let $x$ be a real number and $\{q_n\}$ any sequence of positive numbers.  The inequality
$$p_{n+1}^{x}-p_{n}^{x}<\frac{(p_n p_{n+1})^{x}}{q_{n} p_{n+1}}+p_n^{x}\frac{q_{n+1}-q_{n}}{q_{n}}. $$
holds for infinitely many values of $n$.
\end{corollary}
\begin{proof}
Note that
$$q_n( p_{n+1}^{x}-p_{n}^{x}) =q_np_{n+1}^{x}-q_np_{n}^{x}+q_{n+1}p_{n}^{x}-q_{n+1}p_{n}^{x} <\frac{p_{n}^{x}}{p_{n+1}^{1-x}}-q_np_n^x+q_{n+1}p_n^x. $$

This implies that

$$p_{n+1}^{x}-p_{n}^{x}<\frac{(p_n p_{n+1})^{x}}{q_{n} p_{n+1}}+p_n^{x}\frac{q_{n+1}-q_{n}}{q_{n}}. $$
\end{proof}

The next result present a general bound for gaps of prime numbers which was obtained previously in \cite{azevedo1}.

\begin{corollary}\label{coroaux1}  Let $\{q_n\}$ be a sequence of positive numbers. The inequality
$$p_{n+1}-p_{n}<p_n\frac{q_{n+1}-q_{n}+1}{q_{n}}, $$
holds for infinitely many values of $n$.
\end{corollary}
\begin{proof}
It is enough to take $x=1$ in Theorem \ref{main}.
\end{proof}

In particular, from the previous corollary we have the following behaviour.

\begin{corollary}\label{coroaux2}  If  $x>0$ then 
$$\liminf_{n\to\infty}\frac{p_{n+1}-p_{n}}{p_{n}^{x}}=0 .$$

\end{corollary}
\begin{proof}
Let $x>0$ be fixed and take as $q_n=n$ in  Corollary \ref{coroaux1}.
From the Prime Number Theorem   we have that 
$\lim \frac{p_n}{n\log(n)}=\lim \frac{\log(p_n)}{\log(n)}=1$. Hence
$$ \frac{p_{n+1}-p_{n}}{p_n^{x}}<\frac{p_n}{n\log(n)}\frac{n\log(n)}{p_n^{x}}\frac{2}{n}=2\frac{p_n}{n\log(n)}\frac{\log(n)}{p_n^{x}} $$
holds for infinitely many values of $n$. As so
$$ \liminf_{n\to\infty}\frac{p_{n+1}-p_{n}}{p_n^{x}}\leq 2\liminf_{n\to\infty}\frac{p_n}{n\log(n)}\frac{\log(n)}{p_n^{x}} =0.$$

\end{proof}

The next result is related to the celebrated estimate obtained in \cite{gold} which states that
$$\liminf_{n\to\infty}\frac{p_{n+1}-p_{n}}{\log(p_{n})}=0 .$$

Although we prove a 
weaker version, note that we use quite more elementary methods.
 
\begin{corollary}\label{coroaux3}  If  $x>0$ then 
$$\liminf_{n\to\infty}\frac{p_{n+1}-p_{n}}{\log(p_{n})^{1+x}}=0 .$$

\end{corollary}
\begin{proof}
Let $x>0$ be fixed and again take as $q_n=n$ in  Corollary \ref{coroaux1}.
By the same arguments used in the previous proof we have that
$$\frac{p_{n+1}-p_{n}}{\log(p_n)^{1+x}}<\frac{p_n}{n\log(n)}\frac{n}{\log(n)^{x}}\frac{2}{n}, $$
holds for infinitely many values of $n$.
This implies that
$$\liminf_{n\to\infty}\frac{p_{n+1}-p_{n}}{\log(p_n)^{1+x}}\leq 2\liminf_{n\to\infty}\frac{p_n}{n\log(n)}\frac{1}{\log(n)^{x}}=0. $$

\end{proof}

We close the paper with a result related to the Andrica's conjecture.

\begin{corollary} \label{alpha}  If  $0 \leq x < 1 $ then  
\begin{equation}\liminf_{n\to \infty}p_{n+1}^{x}-p_{n}^{x}=0.
\label{eqmain}\end{equation}
\end{corollary}
\begin{proof}
In Corollary \ref{coroaux} let $q_n=p_n$, for all $n\geq 1$.
Then the inequality
$$p_{n+1}^{x}-p_{n}^{x}<\frac{1}{(p_{n+1}p_n)^{1-x}}+ \frac{p_{n+1}-p_{n}}{p_{n}^{1-x}} , $$
holds for infinitely many values of $n$.  Since $0\leq x < 1$ we have that 
$$\lim_{n\to \infty}\frac{1}{(p_{n+1}p_n)^{1-x}}=0. $$
Also, from  Corollary \ref{coroaux2} note that
$$\liminf_{n\to \infty}\frac{p_{n+1}-p_{n}}{p_{n}^{1-x}}=0,$$
thus 
$$\liminf_{n\to \infty}p_{n+1}^{x}-p_{n}^{x}\leq \liminf_{n\to \infty}\frac{1}{(p_{n+1}p_n)^{1-x}}+ \frac{p_{n+1}-p_{n}}{p_{n}^{1-x}} =0. $$
The proof is concluded.
\end{proof}

\end{document}